\documentclass{amsart}[12pt]
\usepackage{enumerate}
\usepackage{amssymb}
\usepackage[pdfborder={0 0 0}, pdftitle={Representation-graded Bredon homology of elementary abelian 2-groups}, pdfauthor={Markus Hausmann, Stefan Schwede},pdfstartview={FitH}]{hyperref}

\DeclareMathOperator{\Hom}{Hom}

\DeclareMathOperator{\res}{res}
 
\newcommand{\mF}{{\mathbb F}}
\newcommand{\mN}{{\mathbb N}}
\newcommand{\mR}{{\mathbb R}}

\newcommand{\mZ}{{\mathbb Z}}
\newcommand{\Jc}{{\mathcal J}}
\newcommand{\iso}{\cong}
\newcommand{\sm}{\wedge}
\newcommand{\xra}{\xrightarrow}

\newcommand{\td}[1]{\langle #1\rangle}
\renewcommand{\to}{\longrightarrow}

\numberwithin{equation}{section}
\newtheorem{theorem}[equation]{Theorem}
\newtheorem{cor}[equation]{Corollary}
\newtheorem{prop}[equation]{Proposition}
 
\theoremstyle{definition}
\newtheorem{defn}[equation]{Definition}
\newtheorem{rk}[equation]{Remark}
\newtheorem{eg}[equation]{Example}
\newtheorem{construction}[equation]{Construction}

\begin{document}

\title[Representation-graded Bredon homology of elementary abelian 2-groups]
{Representation-graded Bredon homology\\of elementary abelian 2-groups}
\author{Markus Hausmann}
\author{Stefan Schwede}
\email{hausmann@math.uni-bonn.de, schwede@math.uni-bonn.de}
\address{Mathematisches Institut, Universit{\"a}t Bonn, Germany}

\begin{abstract}
  We calculate the representation-graded Bredon homology rings
  of all elementary abelian 2-groups with coefficients in the constant mod-2 Mackey functor.
  We exhibit minimal presentations for these rings
  as quotients of the polynomial algebra on the pre-Euler and inverse Thom classes
  of all nontrivial characters,
  subject to an explicit finite list of relations arising from orientability properties.
  Two corollaries of our presentation are the calculation,
  originally due to Holler and Kriz, of the geometric fixed point rings,
  and a strengthening of a calculation of Balmer and Gallauer
  of the localized twisted cohomology ring.\medskip\\
  2020 MSC: 55N91, 55P91, 55Q91 
\end{abstract}

\maketitle

\section*{Introduction}

In this paper we establish minimal presentations for the representation-graded Bredon homology
rings of all elementary abelian 2-groups,
with coefficients in the constant mod-2 Mackey functor $\underline\mF_2$.
More specifically, we determine the `effective cone' of the $R O$-graded Bredon homology ring,
i.e., the sector given by the reduced Bredon homology groups
of linear representation spheres; we denote this multigraded ring by $H(A,\star)$.
The following is our main result,
to be proved as Theorem \ref{thm:Bredon_RO-graded} below.\smallskip

{\bf Theorem.} {\em
  Let $A$ be an elementary abelian 2-group.
  The $\mF_2$-algebra $H(A,\star)$ is generated
    by the pre-Euler classes $a_\lambda$ and the inverse Thom classes $t_\lambda$
    for all nontrivial $A$-characters $\lambda$.
    The ideal of relations between the classes $a_\lambda$ and $t_\lambda$
    is generated by the polynomials
    \[ \sum_{\lambda\in T}  a_\lambda \cdot \bigg( \prod_{\mu\in T\setminus\{\lambda\}} t_\mu\bigg) \]
    for all minimally dependent sets $T$ of nontrivial $A$-characters.
}\smallskip

In \cite[Theorem 3.5]{hausmann-schwede:bordism}, we use the presentation of
$H(A,\star)$ to establish a `global' universal property of
mod-2 Bredon homology, i.e., of the system of all representation-graded Bredon
homology rings of all elementary abelian $2$-groups, including the functoriality in
group homomorphisms. In the language of \cite{hausmann-schwede:bordism}, Bredon
homology is an initial additively
oriented $\textrm{el}^{R O}_2$-algebra.

The fact that $H(A,\star)$ is generated by the classes $a_\lambda$ and $t_\lambda$
was previously shown by Holler and Kriz \cite{holler-kriz}.
Our main new contribution is determining the relations between these classes.
The relations listed in the theorem are minimal, i.e., none of them can be omitted.
The origin of the relations is the fact that 
for a minimally dependent set $T$ of nontrivial $A$-characters,
the $A$-representation $\bigoplus_{\lambda\in T}\lambda$ is orientable,
see Proposition \ref{prop:x-relation}.

We illustrate our result for elementary abelian 2-groups of small rank.
When $A=C=\{\pm 1\}$ is of order 2, the calculation is classical,
originally due to Stong (unpublished), and reproved by several authors.
In this case there is only one nontrivial character, and no relations;
so $ H(C,\star) =\mF_2[a,t]$ is a polynomial algebra
on the pre-Euler class and the inverse Thom class.
The ring $H(C^2,\star)$ was calculated by Ellis-Bloor in \cite[Theorem 4.14]{ellis-bloor}.
In this case there are three nontrivial characters
$p_1,p_2$ and $\mu$, and all relations are generated by the single relation
\[ a_1 t_2 t_\mu + t_1 a_2 t_\mu + t_1 t_2 a_\mu\ = \ 0\ .\]
So the minimal presentation of $H(C^2,\star)$ is
\[  H(C^2,\star) \ =\ \mF_2[a_1,a_2,a_\mu,t_1,t_2,t_\mu] / (a_1 t_2 t_\mu + t_1 a_2 t_\mu + t_1 t_2 a_\mu)\ .\]

To the best of our knowledge, the presentation of $H(A,\star)$ is new when the rank of $A$
exceeds two.
For $A=C^3$ we make our presentation of $H(C^3,\star)$
completely explicit in Example \ref{eg:rank3}. In this case there are
14 polynomial generators, namely the classes $a_\lambda$ and $t_\lambda$
for each of the seven nontrivial characters $\lambda$, and 14 minimal relations.
Of these relations, seven are cubic in the generators, and of the same general form
as in the previous example, i.e.,
$a_{\alpha} t_{\beta} t_{\gamma} + t_{\alpha} a_{\beta} t_{\gamma} + t_{\alpha} t_{\beta} a_{\gamma}  = 0$
for all triples of distinct nontrivial characters that satisfy  $\alpha\cdot\beta\cdot\gamma=1$.
And there are seven minimal relations that are homogeneous of degree 4 in the generators,
of the form
\[  a_{\alpha} t_{\beta} t_{\gamma} t_\delta\ +\ t_{\alpha} a_{\beta} t_{\gamma} t_\delta \ +\ t_{\alpha} t_{\beta} a_{\gamma}t_\delta\  +\ t_{\alpha} t_{\beta} t_{\gamma}a_\delta \ =\ 0\]
for quadruples of distinct nontrivial characters
that satisfy  $\alpha\cdot\beta\cdot\gamma\cdot\delta=1$.

The number of minimal relations grows very quickly in the rank of the elementary abelian
2-group, see the table in Remark \ref{rk:count_relations}. However, a basic pattern continues
as follows: when $A$ has rank $r$, then a new family of relations appears
that has no predecessor for smaller rank, given by homogeneous polynomials
of degree $r+1$ in the generators.

We use our presentation of $H(A,\star)$ to derive two interesting corollaries.
Inverting the pre-Euler classes $a_\lambda$ for all nontrivial $A$-characters
and restricting to integer gradings yields the $A$-geometric fixed point
ring $\Phi^A_*(H\underline\mF_2)$ of mod-2 Bredon homology. This ring was previously
calculated by Holler and Kriz in \cite[Theorem 2]{holler-kriz},
who also gave a formula for the Poincar{\'e} series of the multigraded ring
$H(A,\star)$ in \cite[Theorem 5]{holler-kriz},
and showed that $H(A,\star)$ maps isomorphically onto
the subring of $\Phi^A_*(H\mF_2)[a_\lambda]$ generated by the classes
$a_\lambda$ and  $t_\lambda=x_\lambda\cdot a_\lambda$ for all nontrivial $A$-characters,
with the notation as in Corollary \ref{cor:holler-kriz}.
We explain in Corollary \ref{cor:holler-kriz} how our presentation
of $H(A,\star)$ yields the Holler-Kriz presentation of  $\Phi^A_*(H\mF_2)$
upon localization.
A noticeable feature is that inverting the pre-Euler classes
makes all polynomial relations of degree at least four redundant.

More generally, we consider a subgroup $B$ of $A$ and determine the `mixed' localization
obtained by inverting the pre-Euler classes of all characters
that restrict nontrivially to $B$,
and the inverse Thom classes of all other characters.
The resulting integer-graded ring $H(A|B)$ previously featured in
the work of Balmer and Gallauer \cite{balmer-gallauer}
on the Balmer spectrum of the tt-category of permutation modules.
We obtain an explicit presentation of the ring 
$H(A|B)$ by generators and relations in Theorem \ref{thm:localization};
here, too, all polynomial relations
of degree at least four in the minimal presentation of $H(A,\star)$
become redundant in the localization.
Our calculation improves \cite[Theorem 8.13]{balmer-gallauer},
for the prime $2$, from a `presentation modulo nilpotence' to
an actual presentation.
The graded ring $H(A,\star)$ is a domain,
see Theorem \ref{thm:domain} and Remark \ref{rk:domain};
 so its localization $H(A|B)$ is a domain, too.
\smallskip

{\bf Acknowledgments.}
The authors are members of the Hausdorff Center for Mathematics
at the University of Bonn (DFG GZ 2047/1, project ID 390685813).
A substantial part of the work for this paper was done while the second author
spent the summer term 2023 on sabbatical at Stockholm University, with financial support 
from the Knut and Alice Wallenberg Foundation;
the second author would like to thank SU for the hospitality
and stimulating atmosphere during this visit.

\section{Representation-graded Bredon homology} 

In this section we review some basic features of Bredon homology in a form adapted
for our purposes. This section does not contain any new mathematics.
What is now called {\em Bredon cohomology} was introduced by Bredon in \cite{bredon}
for finite groups and equivariant CW-complexes. The corresponding equivariant homology
theory was introduced by Illman in \cite{illman:equivariant singular}.
Illman develops the theory for arbitrary topological groups,
and he uses singular chains to define the equivariant homology
and cohomology groups on arbitrary equivariant spaces.
In many ways, Bredon homology and cohomology are the correct generalizations
of singular \mbox{(co-)}homology to the equivariant context,
and of fundamental importance in equivariant topology.

\begin{construction}[Bredon homology]
  We recall one construction of Bredon homology with coefficients in a constant Mackey functor.
  We employ a definition that is naturally isomorphic to the original one of Bredon and Illman,
  namely as the equivariant homology theory represented by the Eilenberg-MacLane $G$-spectrum
  $H\underline M$ of the constant Mackey functor
  $\underline M$ associated to an abelian group $M$.
  In other words, we define the $m$th reduced $G$-equivariant Bredon homology group
  of a based $G$-space $X$ with $\underline M$-coefficients as
  \[ \tilde H_m^G(X;\underline M) \ = \ \pi_m^G(H\underline M\sm X)\ .\]
  The groups $\tilde H^G_*(-;\underline M)$ form an equivariant homology theory.
  In particular, they come with a suspension isomorphism,
  and a based $G$-map gives rise to a long exact sequence featuring the Bredon homology groups
  of source, target and the reduced mapping cone.
  We mostly consider $\underline \mF_2$-coefficients in this paper,
  and we will drop $\underline \mF_2$-coefficients from the notation.

  The commutative $G$-ring spectrum structure of $H\underline\mF_2$ gives
  rise to associative, commutative and bilinear pairings
  \[ \cdot \ : \ \tilde H_m^G(X)\times \tilde H_n^G(Y)\ \to \ \tilde H^G_{m+n}(X\sm Y) \]
  for all based $G$-spaces $X$ and $Y$.
\end{construction}

In this paper a {\em representation} of a finite group is a finite-dimensional
orthogonal representation.
Our results are about the $\underline\mF_2$-Bredon homology groups
of representation spheres, i.e., onepoint compactifications $S^V$ of such representations $V$.
The following proposition collects some well-known general facts about these;
we give proofs for the readers' convenience.
Part~(iii) says that automorphisms of representation spheres
are invisible to the eyes of Bredon homology with  $\underline\mF_2$-coefficients.
This means that we can -- and will -- safely ignore the distinction
between $G$-representations and their isomorphism classes.

\begin{prop}\label{prop:vanish high degrees}
  Let $V$ be a $d$-dimensional representation of a finite group $G$.  
  \begin{enumerate}[\em (i)]
  \item
    The group $\tilde H_k^G(S^V)$ is trivial for $k<0$ and for $k>d$,
    and the restriction homomorphism
    \[ \res^G_1\ :\  \tilde H_d^G(S^V)\ \to \  \tilde H_d(S^V) \]
    is an isomorphism. Hence the $\mF_2$-vector space $\tilde H_d^G(S^V)$ is 1-dimensional.
  \item
    If $V$ is orientable, then the restriction homomorphism for constant integer coefficients
    \[ \res^G_1\ :\  \tilde H_d^G(S^V;\underline{\mZ})\ \to \  \tilde H_d(S^V;\mZ) \]
    is an isomorphism.
    Hence the abelian group $\tilde H_d^G(S^V;\underline\mZ)$ is free of rank 1.
  \item
    For every based $G$-homotopy equivalence $\psi:S^V\to S^V$ and every $m\geq 0$, the map
    \[  \tilde H^G_m(\psi)\ : \ \tilde H^G_m(S^V) \ \to \  \tilde H^G_m(S^V)   \]
    is the identity.
\end{enumerate}
\end{prop}
\begin{proof}
  (i) The groups $\tilde H_*^G(S^V)$ can be calculated from
  the reduced cellular chain complex of a $G$-CW-structure on $S^V$
  by taking first $G$-fixed points and then homology.
  Since the reduced cellular chain complex is concentrated 
  in dimensions 0 through $d$, the vanishing claims follow.
  Because the underlying chain complex calculates the reduced homology
  of the underlying space $S^V$, the kernel of the top differential
  $\delta_d:C_d^{\text{cell}}(S^V;\mF_2)\to C_{d-1}^{\text{cell}}(S^V;\mF_2)$
  is 1-dimensional and necessarily with trivial $G$-action. So the kernel of 
  $\delta_d^G:(C_d^{\text{cell}}(S^V;\mF_2))^G\to (C_{d-1}^{\text{cell}}(S^V;\mF_2))^G$
  is also 1-dimensional, yielding the second claim.
  The argument for constant integer coefficients in (ii) is essentially the same;
  the orientability assumption is equivalent to the condition
  that $G$ acts trivially on the group of $d$-cycles in the underlying chain complex.

  (iii)
  Because $\psi:S^V\to S^V$ is a $G$-homotopy equivalence, its class
  in the $G$-equivariant 0-stem is a unit.
  Because Bredon homology is represented by the commutative $G$-ring spectrum
  $H\underline\mF_2$, the map $\tilde H^G_*(\psi)$ equals multiplication
  by the Hurewicz image of $\td{\psi}$ in $\pi_0^G(H\underline\mF_2)$.
  The ring $\pi_0^G(H\underline\mF_2)$ is isomorphic to $\mF_2$, so 1 is its only unit.
  This proves the claim.
\end{proof}

\begin{construction}[The representation-graded ring $H(G,\star)$]
  We introduce notation to deal with the representation-grading.
  For a finite group $G$, we let $J_G$ be the abelian monoid, under direct sum,
  of isomorphism classes of $G$-representations with trivial $G$-fixed points.
  So $J_G$ is freely generated by the
  isomorphism classes of nontrivial irreducible $G$-representations.
  
  We choose a representative for each element of $J_G$. For $\rho\in J_G$ and $m\geq 0$,
  we define
  \[ H_m(G,\rho)\ = \ \tilde H_m^G(S^V)\ , \]
  where $V$ is the chosen representative of $\rho$.
  Proposition \ref{prop:vanish high degrees} (iii) guarantees that this definition
  is independent of the representative up to preferred isomorphism,
  induced by any isomorphism of representations.
  
  The pairing of Bredon homology induces an associative, commutative and bilinear pairing
  \[ \cdot \ : \ H_m(G,\rho)\times H_n(G,\kappa)\ \to \ H_{m+n}(G,\rho+\kappa) \]
  for $\rho,\kappa\in J_G$ and $m,n\geq 0$.
  This map is defined as the composite
  \[ \tilde H_m^G(S^V)\times \tilde H_n^G(S^W)\  \xra{\ \cdot\ }\
    \tilde H_{m+n}^G(S^V\sm S^W)\ \iso \  \tilde H_{m+n}^G(S^U)\ , \]
  where $V$, $W$ and $U$ are the chosen representatives of $\rho$, $\kappa$
  and $\rho+\kappa$, respectively.
  The last isomorphism is induced by a choice of $G$-equivariant isomorphism
  $V\oplus W\iso U$; it is independent of the choice
  by Proposition \ref{prop:vanish high degrees} (iii).
  We emphasize that this multiplication is strictly commutative, i.e., for
  $x\in H_m(G,\rho)$ and $y\in H_n(G,\kappa)$, the classes
  $x\cdot y$ and $y\cdot x$ are {\em equal} in the group
  $H_{m+n}(G,\rho+\kappa)=H_{n+m}(G,\kappa+\rho)$.
  This, one more time, uses that automorphisms of representation spheres
  are invisible in $\tilde H^G_*$.
  The multiplication maps thus make the collection of groups $H_m(G,\rho)$
  into a commutative $(\mN\times J_G)$-graded $\mF_2$-algebra.
  We denote this object by $H(G,\star)$ and refer to it as the
  {\em representation-graded Bredon homology ring} of the group $G$.
  We will routinely abuse notation by identifying a $G$-representation $V$ with
  trivial fixed points with its class in $J_G$; thus we shall write
  $H_m(G,V)$ for $H_m(G,[V])$.
\end{construction}

\begin{rk}
  Bredon homology with coefficients in a Mackey functor is represented by
  a genuine $G$-spectrum, and hence can be extended to a homology theory
  for $G$-spaces that is $R O(G)$-graded. Our results are only about
  the `effective cone' of the $R O(G)$-graded coefficient ring,
  i.e., the sector given by Bredon homology of representation spheres.
  The effective cone has a much nicer algebraic structure than the rest
  of the $R O(G)$-graded Bredon homology, which tends to contain
  many nilpotent classes and trivial products.
  The effective cone contains the pre-Euler and inverse Thom classes,
  so it determines the geometric fixed points (obtained by inverting all pre-Euler classes),
  and various other localizations, see Construction \ref{con:invert classes} below.
\end{rk}

We recall two kinds of classes that exist for every $G$-representation,
the pre-Euler class and the inverse Thom class.
Our pre-Euler class is also called `Euler class' by other authors.

\begin{construction}[Pre-Euler and inverse Thom classes]
  We let $V$ be  a $G$-representation. The {\em pre-Euler class}
  \[ a_V\ \in \ H_0(G,V) \]
  is the image of the multiplicative unit $1\in H_0(G,0)$
  under the homomorphism induced by the based $G$-map $S^0\to S^V$
  that sends the point $0$ to the $G$-fixed point $0$ in $S^V$.
  The pre-Euler class can be zero, for example if $V$ has nonzero $G$-fixed points.
  
  For $d=\dim(V)$, the {\em inverse Thom class} is the unique nonzero element
  \[ t_V\ \in \ H_d(G,V)\ , \]
  compare Proposition \ref{prop:vanish high degrees} (i).
  If $W$ is another $G$-representation, then
  \[ a_V\cdot a_W\ =\ a_{V\oplus W}\text{\qquad and\qquad} t_V\cdot t_W=t_{V\oplus W}\]
  in $H_*(G,V\oplus W)$.
\end{construction}

In this paper, a {\em character} of a finite group $G$
is a group homomorphism $\lambda:G\to C=\{\pm 1\}$.
We shall routinely confuse a character $\lambda$
with the 1-dimensional $G$-representation on $\mR$ in which $g\in G$
acts by multiplication by $\lambda(g)$.
Elementary abelian 2-groups are characterized among finite groups
by the property that all irreducible real representations are 1-dimensional,
and hence given by characters.

We will make use of the Bockstein homomorphism
\[ \beta\ : \ \tilde H_m^G(X)\ \to \ \tilde H_{m-1}^G(X) \]
associated to the short exact sequence of constant Mackey functors
$\underline\mF_2\to\underline\mZ/4\to\underline\mF_2$.

\begin{eg}
  Let $\lambda:G\to C$ be a nontrivial character with kernel $K$.
  The minimal $G$-CW-structure of $S^\lambda$ with two fixed 0-cells
  and one 1-cell with isotropy $K$ shows that for every abelian group $M$,
  the group $\tilde H_n^G(S^\lambda;\underline M)$ is trivial for $n\ne 0,1$.
  And it yields an exact sequence of Bredon homology groups:
    \begin{align*}
     0 \to  \tilde H^G_1(S^\lambda;\underline M)\to
      H^G_0(G/K;\underline M)\to H^G_0(G/G;\underline M) \xra{\cdot a_\lambda}
      \tilde H_0^G(S^\lambda;\underline M)\to 0
  \end{align*}
  The middle two groups are coefficient groups of the Mackey functor,
  and thus equal to $M$. The middle homomorphism
  is the transfer from $K$ to $G$ in the Mackey functor $\underline M$,
  i.e., multiplication by the index $[G:K]=2$.
  So the groups $\tilde H_n^G(S^\lambda;\underline M)$
  in dimension 1 and 0 are isomorphic to the kernel and cokernel,
  respectively, of multiplication by $2$ on $M$.
  For $M=\mF_2$ we conclude that $H_1(G,\lambda)$ and $H_0(G,\lambda)$
  are 1-dimensional, generated by the inverse Thom class $t_\lambda$
  and the pre-Euler class $a_\lambda$, respectively.

  We claim that the Bockstein homomorphism takes $t_\lambda$ to $a_\lambda$.
  Indeed, the short exact sequence of Mackey functors
  $\underline\mF_2\to\underline\mZ/4\to\underline\mF_2$
  yields an exact sequence of Bredon homology groups:
  \begin{align*}
    \tilde H^G_1(S^\lambda;\underline\mF_2)\ \xra{\ \beta\ }   \tilde H^G_0(S^\lambda;\underline\mF_2)\to
    \tilde H^G_0(S^\lambda;\underline{\mZ}/4)\to
    \tilde H^G_0(S^\lambda;\underline\mF_2)\to 0
  \end{align*}
  Since the four nontrivial groups in this sequence are all cyclic of order 2,
  exactness implies that the Bockstein homomorphism is an isomorphism.
  Since source and target are spanned by $t_\lambda$ and $a_\lambda$, respectively,
  the Bockstein satisfies $\beta(t_\lambda)=a_\lambda$.
\end{eg}

\begin{prop}\label{prop:x-relation}
  Let $G$ be a finite group.
  \begin{enumerate}[\em (i)]
  \item If $V$ is an orientable $G$-representation, then $\beta(t_V)=0$.
  \item Let $T$ be a set of $G$-characters whose product is 1. Then
    \[
      \sum_{\lambda\in T} a_\lambda \cdot \bigg( \prod_{\mu\in T\setminus\{\lambda\}} t_\mu\bigg)
      \ = \ 0\ .\]
\end{enumerate}
\end{prop}
\begin{proof}
  (i) If $V$ is orientable, then by Proposition \ref{prop:vanish high degrees} (ii),
  the inverse Thom class $t_V$ lifts to a class in 
  the Bredon homology with constant integral coefficients.
  So the image of $t_V$ under the Bockstein homomorphism is trivial.

  (ii) We let $V=\bigoplus_{\lambda\in T}\lambda$ be the sum of the characters in the set $T$.
  The determinant of $V$ is the product of the characters in $T$, which is trivial
  by hypothesis. So the $G$-representation $V$ is orientable, and hence $\beta(t_V)=0$
  by part (i).
  The Bockstein homomorphism is a derivation, in the sense that
  $\beta(x\cdot y) =  \beta(x)\cdot y + x\cdot \beta(y)$
  for all classes $x\in\tilde H_m^G(X)$ and $y\in\tilde H_n^G(Y)$.
  Applying the derivation property repeatedly and using that $\beta(t_\lambda)=a_\lambda$ shows that
\[  0 \ = \ \beta(t_V) \ = \  \beta\big({\prod}_{\lambda\in T}t_\lambda\big) \ = \
  {\sum}_{\lambda\in T}\ a_\lambda\cdot \big( {\prod}_{\mu\in T\setminus\{\lambda\}} t_\mu\big) \ .
  \vspace*{-.5cm} \]
\end{proof}

\section{Bredon homology for elementary abelian 2-groups}

In this section we specialize from general finite groups to elementary abelian
$2$-groups, and we prove our main result, Theorem \ref{thm:Bredon_RO-graded}.
There we exhibit a presentation of $H(A,\star)$ as the quotient of a
polynomial $\mF_2$-algebra on the classes $a_\lambda$ and $t_\lambda$
for all nontrivial $A$-characters $\lambda$,
by an explicit minimal set of homogeneous polynomial relations.
Along the way, we give an elementary and self-contained proof
that the ring $H(A,\star)$ is a domain, see Theorem \ref{thm:domain}.

Many of our arguments involve bootstrapping information about Bredon homology
of a subgroup to the ambient group.
In those arguments, we need to restrict representations to subgroups,
which typically creates new fixed points.
If $V$ is a $G$-representation with $V^G=0$, and $K$ a subgroup of $G$, 
we set $k=\dim(V^K)$, and we let $V_K=V-V^K$ be the orthogonal complement
of the $K$-fixed points in $V$.
The {\em restriction homomorphism}
\[ \res^G_K \ : \ H_m(G,V)\ \to \ H_{m-k}(K,V_K) \]
is the composite
\[ \tilde H^G_m(S^V)\ \xra{\res^G_K} \
  \tilde H^K_m(S^V)\ \iso \
  \tilde H^K_m(S^{V_K}\sm S^k)\ \xra[\iso]{(-\sm S^k)^{-1}} \
  \tilde H^K_{m-k}(S^{V_K})\ . \]
The first isomorphism is induced by a choice of $K$-equivariant
isomorphism $V\iso V_K\oplus\mR^k$; it is independent of this choice by
Proposition \ref{prop:vanish high degrees} (iii).
The second isomorphism is the inverse of the suspension isomorphism.
The restriction homomorphism is multiplicative, and its effect
on inverse Thom and pre-Euler classes is given by
\[ \res^G_K(t_V)\ = \ t_{V_K} \text{\quad and\quad}
 \res^G_K(a_V)\ = \
  \begin{cases}
    a_{V_K} & \text{ if $V^K=0$, and}\\
    \ 0 & \text{ if $V^K\ne 0$.}
  \end{cases}
\]

\begin{prop}\label{prop:fundamental}
  Let $A$ be an elementary abelian 2-group and let $W$ be an $A$-representation
  with trivial fixed points.
  \begin{enumerate}[\em (i)]
\item For every subgroup $B$ of $A$, the restriction homomorphism
  \[ \res^A_B\ : \ H_m(A,W)\ \to \ H_{m-k}(B,W_B) \]
  is surjective, where $k=\dim(W^B)$.
\item 
  Suppose that $W=V\oplus\lambda$ for an $A$-representation $V$ and
  a nontrivial $A$-character $\lambda$ with kernel $K$.
  Then for $k=\dim(V^K)$, the following sequence is exact: 
  \[ 0 \ \to \ H_m(A,V)\ \xra{\ \cdot a_\lambda} \ H_m(A,V\oplus\lambda)\ \xra{\res^A_K} \
    H_{m-k-1}(K,V_K)\ \to\ 0  \]

  \item The  $\mF_2$-vector space $H_m(A,W)$ is spanned by the classes $a_U\cdot t_V$
  for all $A$-representations $U$ and $V$ such that $U\oplus V= W$
  and $m=\dim(V)$.
\end{enumerate}
\end{prop}
\begin{proof}
  We prove all three statements together by induction over the rank of $A$.
  The induction starts when $A$ is the trivial group, in which case there is nothing to show.
  Now we let $A$ be a nontrivial elementary abelian 2-group, and we assume
  that parts (i)--(iii) hold for all proper subgroups of $A$.

  We start by proving (i), where we may assume that $B$ is a proper subgroup of $A$.
  By part (iii) for $B$, it suffices to show that all classes of the form
  $a_U\cdot t_V$ are in the image of the restriction homomorphism,
  whenever $U$ and $V$ are $B$-representations such that $U\oplus V=W_B$
  and $m-k=\dim(V)$. Because $W$ is a sum of 1-dimensional $A$-representations,
  we may choose an $A$-equivariant decomposition $W=\bar U\oplus\bar V\oplus T$
  such that  $\res^A_B(\bar U)\iso U$, $\res^A_B(\bar V)\iso V$,
  and $B$ acts trivially on $T$.
  Then $\res^A_B(a_{\bar U}\cdot t_{\bar V}\cdot t_T)= a_U\cdot t_V$,
  and we have shown part (i) for $A$.

  Now we prove (ii).
  Smashing the cofiber sequence of based $A$-spaces
  \[ A/K_+ \ \xra{\quad} \ S^0 \ \xra{\quad }\ S^\lambda \ \to\ A/K_+\sm S^1 \]
  with $S^V$ and applying $A$-equivariant Bredon homology yields a long exact sequence: 
  \[ 
       \dots\ \to \  H_m(A,V) \  \xra{\ \cdot a_\lambda}\ H_m(A,V\oplus\lambda)\
     \xra{\ \partial\ }\  \tilde H_{m-1}^A(S^V\sm A/K_+)\ \to\  \dots  
 \]
  The Wirthm{\"u}ller and suspension isomorphisms identify the group $\tilde H_{m-1}^A(S^V\sm A/K_+)$
  with $\tilde H_{m-k-1}^K(S^{V_K})=H_{m-k-1}(K,V_K)$.
  Under this identification, the boundary map $\partial$ becomes the
  restriction homomorphism $\res^A_K:H_m(A,V\oplus\lambda)\to H_{m-k-1}(K,V_K)$,
  which is surjective by (i).
  So the long exact sequence decomposes into short exact sequences, showing (ii).

  We prove (iii) by induction on the dimension of $W$.
  For $W=0$, the groups $H_*(A,0)$ consist of a single copy of $\mF_2$
  in dimension 0, spanned by the multiplicative unit $1=t_0$.
  If $W$ is nonzero, we write $W=V\oplus\lambda$
  for an $A$-representation $V$ and a nontrivial $A$-character $\lambda$, with kernel $K$.
  By part (i), the restriction map  $\res^A_K : H_{m-1}(A,V)\to  H_{m-k-1}(K,V_K)$ is surjective.
  So for every class $x\in H_m(A,W)$,
  there is a class $z\in H_{m-1}(A,V)$ such that $\res^A_K(z)=\res^A_K(x)$. Then
  \[ \res^A_K(x + z\cdot t_\lambda)\ = \ \res^A_K(x) + \res^A_K(z)\ = \ 0 \ .\]
  Part (ii) provides a class $y\in H_m(A,V)$ such that  $y\cdot a_\lambda= x + z\cdot t_\lambda$.
  Because the dimension of $V$ is smaller than that of $W$,
  the classes $y$ and $z$ are sums of products of $a$-classes and $t$-classes.
  Hence the same is true for $x=y\cdot a_\lambda+ z\cdot t_\lambda$,
  and we have shown (iii). 
\end{proof}

Our next result shows that the rings $H(A,\star)$ have no zero-divisors.
As we explain in Remark \ref{rk:domain} below, this can also be proven
by combining results of \cite{holler-kriz} and \cite{skriz},
so we make no claim to originality.
As a service to the reader, we record this key structural result explicitly and
give an independent, self-contained and elementary proof.

\begin{theorem}\label{thm:domain}
  For every elementary abelian 2-group $A$, the representation-graded Bredon homology
  ring $H(A,\star)$ is a domain.
\end{theorem}
\begin{proof}
  We call an element of $H(A,\star)$ {\em regular} if multiplication by it is injective.
  We will show that all nonzero homogeneous elements of $H(A,\star)$ are regular.
  We argue by induction over the rank of $A$.
  For $A=\{1\}$, the ring is the field $\mF_2$, hence a domain.

  Now we suppose that $A$ is nontrivial.
  Proposition \ref{prop:fundamental} (ii) shows that the pre-Euler classes
  of all nontrivial $A$-characters are regular.
  We show next that all the inverse Thom classes $t_\mu$ are regular.
  So we let $y\in H_n(A,W)$ be a homogeneous element such that $t_\mu\cdot y=0$,
  for some $A$-representation $W$ with trivial fixed points.
  We argue by induction on the dimension of $W$.
  If $W=0$ there is nothing to show because the integer-graded part of Bredon homology
  is spanned by the multiplicative unit, and $t_\mu\ne 0$.
  For $W\ne 0$ we write $W=U\oplus\lambda$ for some nontrivial $A$-character $\lambda$,
  with kernel $K$.
  Then 
  \[  \res^A_K(t_\mu)\cdot \res^A_K(y) \  =\ \res^A_K(t_\mu\cdot y)\ =\ 0\ .\]
  Because $\res^A_K(t_\mu)$ is either 1 (if $\lambda=\mu$),
  or the inverse Thom class of the restricted character $\mu|_K$
  (if $\lambda\ne\mu$), and because $K$ has smaller rank than $A$,
  the class $\res^A_K(t_\mu)$ is nonzero and regular, so $\res^A_K(y)=0$.
  Proposition \ref{prop:fundamental} (ii)
  provides a class $u\in H_n(A,U)$ such that $y= u\cdot a_\lambda$.
  Then $t_\mu\cdot u\cdot a_\lambda=t_\mu\cdot y=0$, so
  $t_\mu\cdot u=0$  because $a_\lambda$ is regular.
  Because $U$ has smaller dimension than $W$,
  we deduce that $u=0$.  Hence also $y=0$, and this concludes the special case.

  Now we show that a general nonzero homogeneous element $x\in H_m(A,V)$ is regular,
  where $V$ is an $A$-representation with trivial fixed points.
  The assumption that $x$ is nonzero implies that $m\leq \dim(V)$.
  We argue by induction on $\dim(V)-m$.
  If $m=\dim(V)$, then $x=t_V$, and so $x$ is a product of inverse Thom classes
  of $A$-characters. Since all the factors are regular by the special case, so is $t_V$.

  Now we suppose that $m<\dim(V)$. We distinguish two cases.
  In the first case we suppose that there is an index 2 subgroup $K$ of $A$
  such that $\res^A_K(x)=0$. We let $\lambda$ be the $A$-character whose kernel is $K$.
  Then also $\res^A_K(x\cdot t_\lambda)=0$.
  Proposition \ref{prop:fundamental} (ii)
  provides a class $u\in H_{m+1}(A,V)$ such that $u\cdot a_\lambda=x\cdot t_\lambda$.
  Since $x\ne 0$ and $t_\lambda$ is regular, the class $u$ is nonzero.
  Because the integer degree of $u$ is larger than that of $x$, 
  the induction hypothesis shows that $u$ is regular. Since $u$, $a_\lambda$ and $t_\lambda$
  are regular, so is $x$.

  In the second case we suppose that for every index 2 subgroup $K$ of $A$,
  the restriction $\res^A_K(x)$ is nonzero.
  We consider a homogeneous element $y\in H_n(A,W)$ such that $x\cdot y=0$,
  where $W$ is another $A$-representation with trivial fixed points.
  We perform another induction on the dimension of $W$.
  If $W=0$, then $y$ lies in the integer-graded subring, so $y=0$ or $y=1$, and we are done.
  For $W\ne 0$ we write $W=U\oplus\lambda$ for some nontrivial $A$-character $\lambda$,
  with kernel $K$. Then
  \[ \res^A_K(x)\cdot \res^A_K(y)\ = \ \res^A_K(x\cdot y)\ = \ 0\ . \]
  Because $\res^A_K(x)\ne 0$ and $H(K,\ast)$ is a domain by induction,
  this implies $\res^A_K(y)=0$.
  Proposition \ref{prop:fundamental} (ii) provides
  a class $u\in H_n(A,U)$ such that $y=u\cdot a_\lambda$.
  Thus $x\cdot u\cdot a_\lambda=x\cdot y=0$.
  So $x\cdot u=0$ because $a_\lambda$ is regular.
  Because the dimension of $U$ is smaller than the dimension of $W$,
  we conclude that $u=0$, and hence also $y=0$.
\end{proof}

Now we move on to our main result, the minimal presentation
of the representation-graded Bredon homology ring $H(A,\star)$.
Proposition \ref{prop:fundamental} (iii) shows that $H(A,\star)$
is generated as an $\mF_2$-algebra
by the classes $a_\lambda$ and $t_\lambda$ for all nontrivial $A$-characters $\lambda$.
If $A=C$, there is only one nontrivial character, and
then $H(C,\star)$ is well known to be a polynomial algebra
on the classes $a$ and $t$.
If $A$ is elementary abelian of rank at least 2,
however, there are nontrivial polynomial relations between the
pre-Euler classes and the inverse Thom classes.

In the following, it will be convenient to use the notation
\[ A^\circ \ = \ \Hom(A,C)\setminus\{1\} \]
for the set of nontrivial characters of an elementary abelian 2-group.

\begin{defn}[Minimally dependent sets of characters]
  Let $A$ be an elementary abelian 2-group.
  A nonempty subset of $A^\circ$ is {\em dependent}
  if it is linearly dependent as a subset of the vector space $\Hom(A,C)$.
  The set is {\em minimally dependent} if it is dependent, but no proper subset is dependent.
\end{defn}

\begin{rk}\label{rk:cardinality_minimal}
If $T$ is a minimally dependent subset of $A^\circ$,
then the product of all its elements must be the trivial character.
Indeed, the linear dependence of $T$ means that
some subset of it has product the trivial character;
if this were the case for some proper subset of $T$,
then that subset would be dependent, contradicting minimality.

Sets of nontrivial $A$-characters that have one or two elements
have a nontrivial product.
So every dependent set of $A$-characters has a least three elements.
If $A$ has rank $r$, then every set of $r+1$ nontrivial $A$-characters is dependent,
and hence has a nonempty subset whose product is the trivial character.
So a minimally dependent set of $A$-characters has at most $r+1$ elements.

The group $C^2$ has precisely one minimally dependent set of nontrivial characters,
the set of all three nontrivial characters.
In Example \ref{eg:rank3} we enumerate all 14 minimally dependent sets of nontrivial characters
of the group $C^3$.
In Remark \ref{rk:count_relations} we determine the number of minimally dependent subsets
of $A^\circ$ as a function of the rank of $A$.
\end{rk}

The next theorem is our main result about representation-graded Bredon homology,
providing an explicit set of homogeneous polynomial relations between the
pre-Euler and inverse Thom classes,
parameterized by minimally dependent sets of characters.
Moreover, these relations form a minimal generating set for the ideal of all relations.
In \cite[Theorem 3.5]{hausmann-schwede:bordism}, we use these presentations of
$H(A,\star)$ to establish a `global' universal property of the collection of representation-graded
Bredon homology rings: we show that mod-2 Bredon homology is an initial additively
oriented $\textrm{el}^{R O}_2$-algebra.
Holler and Kriz determined the Poincar{\'e} series of the multigraded ring
$H(A,\star)$ in \cite[Theorem 5]{holler-kriz}.
We have not investigated how to derive their formula for the 
Poincar{\'e} series from our presentation.

\begin{theorem}\label{thm:Bredon_RO-graded}
  Let $A$ be an elementary abelian 2-group.
  \begin{enumerate}[\em (i)]
  \item
    The representation-graded Bredon homology ring $H(A,\star)$
    is generated as an $\mF_2$-algebra
    by the classes $a_\lambda$ and $t_\lambda$ for all nontrivial $A$-characters $\lambda$.
\item 
  The kernel of the surjective homomorphism of $\mF_2$-algebras
  \[ \epsilon_A \ : \  \mF_2[a_\mu,t_\mu\colon \mu\in A^\circ] \ \to \ H(A,\star) \]
  is the ideal generated by the polynomials
  \[  r(T)\ = \
    \sum_{\lambda\in T}  a_\lambda \cdot\bigg( \prod_{\mu\in T\setminus\{\lambda\}}\, t_\mu\bigg) \]
  for all minimally dependent subsets $T$ of $A^\circ$.
\item
  Every set of homogeneous elements that generates the kernel
  of $\epsilon_A$ contains the polynomials $r(T)$
 for all minimally dependent subsets $T$ of $A^\circ$.
\end{enumerate}
\end{theorem}
\begin{proof}
  Part (i) was shown in Proposition \ref{prop:fundamental} (iii),
  and is repeated here for easier reference.

  (ii)
  For the course of the proof we write $I(A)$ for the ideal of the polynomial ring 
  $\mF[a_\mu,t_\mu\colon \mu\in A^\circ]$ generated by the polynomials $r(T)$
  for all minimally dependent subsets $T$ of $A^\circ$. 
  Since minimally dependent sets of characters multiply to 1,
  Proposition \ref{prop:x-relation} (ii) shows that $I(A)\subseteq \ker(\epsilon_A)$;
  so it remains to show the reverse inclusion.

  We argue by induction on the rank of $A$.
  The induction starts with the trivial group, where there is nothing to show.
  Now we let $A$ be a nontrivial elementary abelian 2-group, and we assume part (ii)
  for all proper subgroups of $A$.
  In the inductive step, we shall make use of the polynomials
  \[  r(T)\ = \
    \sum_{\lambda\in T}  a_\lambda \cdot\bigg( \prod_{\mu\in T\setminus\{\lambda\}}\, t_\mu\bigg) \]
  for arbitrary subsets $T$ of $A^\circ$, not necessarily minimally dependent;
  we alert the reader that if the elements of $T$ do not multiply to the trivial character,
  then the polynomial $r(T)$ will {\em not} map to 0 under $\epsilon_A$.
  
  We let $\lambda$ be a nontrivial $A$-character, with kernel $K$.
  We let $\Jc$ denote the homogeneous ideal in the ring
  $\mF_2[a_\mu,t_\mu\colon  \mu\in A^\circ\setminus\{\lambda\}]$ consisting of those
  elements $y$ such that $\res^A_K(\epsilon_A(y))=0$.
  We emphasize that elements of $\Jc$ are polynomials that do not involve the variables
  $a_\lambda$ nor $t_\lambda$.
  We shall prove two properties of this ideal:\smallskip

  (a) The ideal $\Jc$ is generated by the polynomials $r(T)$ for all
  subsets $T$ of $A^\circ\setminus\{\lambda\}$
  such that either $T$ is minimally dependent, or $T\cup\{\lambda\}$
  is minimally dependent.\smallskip

  (b) In the graded ring   $\mF_2[a_\mu,t_\mu\colon  \mu\in A^\circ]$,
  the relation $t_\lambda\cdot \Jc\subset \td{I(A),a_\lambda}$ holds.\smallskip

  Proof of (a). We write $J_{A\backslash\lambda}$ for the submonoid of $J_A$ consisting of
  the classes of $A$-representations $V$ that do not involve the character $\lambda$.
  Equivalently, $J_{A\backslash\lambda}$ contains the $A$-representations
  with trivial $A$-fixed points such that $V^K=0$.
  The restriction homomorphism $\res^A_K:\mN\times J_{A\backslash\lambda}\to\mN\times J_K$
  lets us inflate $(\mN\times J_K)$-graded rings $R$
  to $(\mN\times J_{A\backslash\lambda})$-graded rings $(\res^A_K)^*(R)$ by setting
  \[ (\res^A_K)^*(R)(k,V)\ = \ R(k,V|_K)\ . \]
  We alert the reader that a grading-inflated polynomial algebra
  is no longer a polynomial algebra.
  With this grading convention, a morphism
  of $(\mN\times J_{A\backslash\lambda})$-graded $\mF_2$-algebras
  \[  \rho^A_K\ : \   \mF_2[a_\mu,t_\mu\colon  \mu\in A^\circ\setminus\{\lambda\}] \ \to \
  (\res^A_K)^*\left(  \mF_2[a_\nu,t_\nu\colon  \nu\in K^\circ]  \right) \]
  is given by sending $a_\mu$ to $a_{\mu|K}$ and $t_\mu$ to $t_{\mu|K}$. 
  Moreover, the ideal $\Jc$ is precisely the kernel of the composite homomorphism
  \begin{align*}
    \mF_2[a_\mu,t_\mu\colon  \mu\in A^\circ\setminus\{\lambda\}] \
    &\xra{\ \rho^A_K\ } \
      (\res^A_K)^*\left(  \mF_2[a_\nu,t_\nu\colon  \nu\in K^\circ]  \right)\\
    &\xra{(\res^A_K)^*(\epsilon_K)} (\res^A_K)^* (  H(K,\star) ) \ .
  \end{align*}
  The kernel of $\rho^A_K$ is the ideal generated by the homogeneous polynomials
  \[
    a_\mu t_{\mu\lambda} + t_\mu a_{\mu\lambda}\ = \ r(\{\mu,\mu\lambda\})
  \]
  for all $\mu\in A^\circ\setminus\{\lambda\}$.
  The set $\{\mu,\mu\lambda\}$ is independent
  and $\{\mu,\mu\lambda\}\cup\{\lambda\}$ is minimally dependent.
  So the polynomials $r(\{\mu,\mu\lambda\})$ are among those of the second kind specified in (a).

  By the inductive hypothesis for the subgroup $K$,
  the kernel of $\epsilon_K: \mF_2[a_\nu,t_\nu\colon  \nu\in K^\circ] \to H(K,\star)$
  is generated by the polynomials $r(S)$ for all minimally dependent subsets $S$
  of $K^\circ$. So the kernel of the homomorphism
  \[  (\res^A_K)^*(\epsilon_K)\ : \ 
    (\res^A_K)^*\left(  \mF_2[a_\nu,t_\nu\colon  \nu\in K^\circ]  \right)
    \ \to \  (\res^A_K)^* (  H(K,\star) ) \]
  is the ideal generated by the same polynomials $r(S)$, but each occurring
  multiple times in different degrees, namely for all subsets $T$ of
  $A^\circ\setminus\{\lambda\}$ such that 
  \[ {\bigoplus}_{\mu\in T}\ \mu|_K \ = \ {\bigoplus}_{\nu\in S}\ \nu\ . \]
  This condition means that for each character $\nu\in S$, the set $T$
  contains exactly one of the two extensions of $\nu$ to an $A$-character.
  Because $S$ is a minimally dependent subset of $K^\circ$, the $A$-character
  \[ {\prod}_{\mu\in T}\ \mu \quad \in \ \Hom(A,C)\]
  then restricts to the trivial character on the subgroup $K$.
  Hence this product is either 1 or $\lambda$.
  In the first case, $T$ is a minimally dependent subset of $A^\circ\setminus\{\lambda\}$.
  In the second case, $T$ is an independent subset of $A^\circ\setminus\{\lambda\}$
  such that $T\cup\{\lambda\}$ is minimally dependent.

  We have now shown that some of the polynomials $r(T)$ with $T$ as specified in (a)
  generate the kernel of $\rho^A_K$, and the images of the others under $\rho^A_K$
  generate the kernel of $(\res^A_K)^*(\epsilon_K)$. So all those polynomials
  together generate the kernel of the composite, and hence the ideal $\Jc$.
  This completes the proof of (a).\medskip

    Proof of (b). It suffices to show that the two types of generating polynomial for
  $\Jc$ specified in (a) have the desired property.
  If $T$ is a minimally dependent subset of $A^\circ\setminus\{\lambda\}$,
  then the polynomial $r(T)$ belongs to the ideal $I(A)$.
  Hence also  $t_\lambda\cdot r(T)\in I(A)$.
  If $T$ is a subset of $A^\circ\setminus\{\lambda\}$ such that
  $T\cup\{\lambda\}$ is minimally dependent, then the relation
  \[  t_\lambda \cdot r(T)\
    = \   r(T\cup\{\lambda\}) \ + \ a_\lambda \cdot {\prod}_{\mu\in T}\ t_\mu  \]
  holds in the ring $\mF_2[a_\mu,t_\mu\colon \mu\in A^\circ]$,
  and witnesses that the left hand side lies in the ideal $\td{I(A),a_\lambda}$.
  This completes the proof of (b).\medskip
  
  Now we complete the inductive step, showing that $\ker(\epsilon_A)\subseteq I(A)$.
  We prove this for all homogeneous pieces $H_m(A,W)$,
  where $W$ is an $A$-representation with $W^A=0$,
  by induction on the dimension of $W$.
  For $W=0$ we are considering integer degrees,
  where source and target of $\epsilon_A$ both
  consist only of a copy of $\mF_2$ in degree 0.
  
  Now we suppose that $W\ne 0$. 
  We let $f$ be a homogeneous polynomial in $\mF_2[a_\mu,t_\mu\colon  \mu\in A^\circ]$
  of degree $(k,W)$ with $\epsilon_A(f)=0$.
  We choose an $A$-character $\lambda$ and
  an $A$-representation $V$ with $W=V\oplus\lambda$.
  We let $K$ denote the kernel of $\lambda$.
  We write $f=t_\lambda\cdot y+a_\lambda\cdot z$ for some homogeneous polynomials $y,z$
  of degrees $(k-1,V)$ and $(k,V)$, respectively. Then
  \[
    \res^A_K(\epsilon_A(y))\ = \
    \res^A_K( t_\lambda\epsilon_A(y) + a_\lambda \epsilon_A(z))
    \ = \     \res^A_K( \epsilon_A(f))    \ = \ 0 \ . \]
  
  Case 1: The $A$-character $\lambda$ occurs in $W$ with multiplicity at least 2.
  Then there is an $A$-representation $U$ such that $V=U\oplus \lambda$.
  Because $\res^A_K(\epsilon_A(y))=0$, there is a class $u\in H_{k-1}(A,U)$
  such that $\epsilon_A(y)=a_\lambda\cdot u$.
  Then 
  \[ a_\lambda\cdot (t_\lambda\cdot u+\epsilon_A(z))\ =\
    t_\lambda\cdot\epsilon_A(y)+a_\lambda\cdot\epsilon_A(z)\ = \ \epsilon_A(f)\ =\ 0\ . \]
  Because multiplication by $a_\lambda$ is injective,
  we deduce that $t_\lambda\cdot u=\epsilon_A(z)$.
  We choose a homogeneous polynomial $g$ in $\mF_2[a_\mu,t_\mu\colon\mu\in A^\circ]$
  of degree $(k-1,U)$ such that $\epsilon_A(g)=u$.
  Then $y+a_\lambda g$ and $z+t_\lambda g$ lie in the kernel of $\epsilon_A$.
  Since $U\oplus\lambda=V$ has smaller dimension than $W$,
  the classes $y+a_\lambda g$ and $z+t_\lambda g$
  are in the ideal $I(A)$, by induction. So
  \[  f\ =\ t_\lambda\cdot y+a_\lambda\cdot z\ = \
    t_\lambda\cdot(y+a_\lambda g) +a_\lambda\cdot(z+t_\lambda g)  \]
  also lies in the ideal $I(A)$.\smallskip

  Case 2: The $A$-character $\lambda$ occurs in $W$ with multiplicity 1.
  Then $\lambda$ does not occur in $V$, and so the polynomial $y$
  does not involve the variable $a_\lambda$ nor $t_\lambda$.
  Because $\res^A_K(\epsilon_A(y))=0$, the polynomial $y$ belongs to the ideal $\Jc$.
  By property (b) of the ideal $\Jc$, the class $t_\lambda \cdot y$ then belongs
  to the ideal generated by $I(A)$ and $a_\lambda$.
  Hence there is a homogeneous polynomial $h$ of degree $(k,V)$
  such that $t_\lambda\cdot y$ is congruent to $a_\lambda\cdot h$ modulo the ideal $I(A)$.
  So the polynomial $f$ in the kernel of $\epsilon_A$ satisfies
  \[ f\ = \ t_\lambda\cdot y + a_\lambda\cdot z\ \equiv \ a_\lambda\cdot (h+z)  \text{\qquad modulo $I(A)$.}\]
  Because $I(A)$ is contained in the kernel of $\epsilon_A$, we deduce the relation
  \[ a_\lambda\cdot\epsilon_A(h+z) \ = \  \epsilon_A(a_\lambda\cdot(h+z)) \ = \ \epsilon_A(f) \ = \  0 \ .\]
  Because multiplication by $a_\lambda$ is injective, we conclude that $\epsilon_A(h+z)=0$.
  Since $V$ has smaller dimension than $W=V\oplus \lambda$, the class $h+z$
  lies in the ideal $I(A)$ by induction. So also $a_\lambda(h+z)$, and hence the class $f$,
  lies in the ideal $I(A)$. This finishes the inductive step, and hence the proof of part (ii).

  (iii)
  Source and target of the homomorphism $\epsilon_A$
  are graded by the abelian monoid of isomorphism classes of $A$-representations.
  This abelian monoid is free on the classes of the $A$-characters
  and thus admits a compatible partial order by declaring $V\leq W$
  if $V$ is isomorphic to a direct summand of $W$.
  Hence all products of a homogeneous polynomial $f$ with other homogeneous elements of
  $\mF_2[a_\mu,t_\mu\colon \mu\in A^\circ]$ have degrees greater or equal than that of $f$.
  By (ii), the kernel of $\epsilon_A$ is generated in gradings 
  that contain the sum of a linearly dependent set of characters.
  So the kernel of $\epsilon_A$ is nontrivial only in degrees 
  that contain the sum of a linearly dependent set of characters.
  The generating relations $r(T)$ specified in (ii) lie in degrees
  that are minimal with this property, and they are the unique nontrivial
  elements in the kernel of $\epsilon_A$ in their degrees.
  So each of the relations $r(T)$ specified in (ii) is necessary to generate
  the kernel of $\epsilon_A$.
\end{proof}

We take the time to go through the presentation of
Theorem \ref{thm:Bredon_RO-graded} for elementary abelian 2-groups of rank at most 3.
For the group $C$ with two elements, the representation-graded Bredon homology
ring is well studied, and a polynomial algebra on the classes $a$ and $t$,
\[ H(C,\star)\ = \ \mF_2[a,t]\ . \]
This calculation is originally due to Stong (unpublished)
and reproved by several authors;
the earliest published reference we know of is \cite[Section 2]{lewis:R O_projective}.

The ring $H(C^2,\star)$ was calculated by
Ellis-Bloor \cite[Theorem 4.14]{ellis-bloor}.
The group $C^2$ has precisely one dependent set of nontrivial characters,
the set $\{p_1,p_2,\mu\}$ of all three nontrivial characters,
and this set is  minimally dependent.
So the presentation of $H(C^2,\star)$ given by Theorem 
\ref{thm:Bredon_RO-graded} has only one relation:
the map from $\mF_2[a_1,a_2,a_\mu,t_1,t_2,t_\mu]$
that takes the polynomial generators to the classes with the same names
factors through an isomorphism of $\mF_2$-algebras
\[ \mF_2[a_1,a_2,a_\mu,t_1,t_2,t_\mu] / (a_1 t_2 t_\mu + t_1 a_2 t_\mu + t_1 t_2 a_\mu) \ \iso \ H(C^2,\star)\ .\]

\begin{eg}[Rank 3]\label{eg:rank3}
  As a proof of concept, we make the minimal presentation
  given by Theorem \ref{thm:Bredon_RO-graded} completely explicit for $A=C^3$.
  In this case there are 14 polynomial generators, namely the classes
  $a_\lambda$ and $t_\lambda$ for each of the seven nontrivial characters $\lambda$,
  and there is also a total of 14 minimal relations.
  Minimally dependent subsets for $C^3$ have either three or four elements.
  There are seven minimally dependent subsets of cardinality 3;
  if $p_i$ denotes the projection of $C^3$ to the $i$th factor, then these sets are
    \begin{align*}
    &\{ p_1,\ p_2,\ p_1 p_2\},\quad
    \{ p_1,\ p_3,\ p_1 p_3\}, \quad \{ p_2,\ p_3,\ p_2 p_3\},\quad \{ p_1,\ p_2 p_3,\ p_1 p_2 p_3\},\\
    &    \{ p_2,\ p_1 p_3,\ p_1 p_2 p_3\},\quad
      \{ p_3,\ p_1 p_2,\ p_1 p_2 p_3\} \text{\quad and\quad}  \{ p_1 p_2,\ p_1 p_3,\ p_2 p_3\} \ .
  \end{align*}
  This gives seven minimal relations of the form
  \[  a_{\alpha} t_{\beta} t_{\gamma}\ +\ t_{\alpha} a_{\beta} t_{\gamma}\ +\ t_{\alpha} t_{\beta} a_{\gamma}\]
  where  $\{\alpha,\beta,\gamma\}$ runs over the above seven sets.
  And there are also seven minimally dependent subsets with four elements, namely 
    \begin{align*}
    &\{ p_1,\ p_2,\ p_3,\ p_1 p_2 p_3\},\quad
    \{ p_1,\ p_2,\ p_1 p_3,\ p_2 p_3\},\quad
    \{ p_1,\ p_3,\ p_1 p_2,\ p_2 p_3\},\\
    & \{ p_2,\ p_3,\ p_1 p_2,\ p_1 p_3\},\quad
      \{ p_1,\ p_1 p_2,\ p_1 p_3,\ p_1 p_2 p_3\},\\
   &\{ p_2,\ p_1 p_2,\ p_2 p_3,\ p_1 p_2 p_3\}\text{\quad and\quad}
    \{ p_3,\ p_1 p_3,\ p_2 p_3,\ p_1 p_2 p_3\}\ .
  \end{align*}
  Each such set $\{\alpha,\beta,\gamma,\delta\}$ gives a minimal relation of the form
  \[  a_{\alpha} t_{\beta} t_{\gamma} t_\delta\ +\ t_{\alpha} a_{\beta} t_{\gamma} t_\delta \ +\ t_{\alpha} t_{\beta} a_{\gamma}t_\delta\  +\ t_{\alpha} t_{\beta} t_{\gamma}a_\delta \ .\]
\end{eg}

\begin{rk}[The number of minimally dependent sets]\label{rk:count_relations}
  We let $A$ be an elementary abelian 2-group of rank $r$.
  As we explained in Remark \ref{rk:cardinality_minimal}, minimally dependent subsets of $A^\circ$
  have at least 3 elements, and at most $r+1$ elements.
  We shall now count how many of these there are. We consider $2\leq k\leq r$.
  Every minimally dependent subset with $k+1$ elements is obtained from
  a linearly independent $k$-element subset $S$ of $A^\circ$ by adding to it
  the product of all its members $\sigma=\prod_{\lambda\in S}\lambda$.
  The resulting minimally dependent subset $S\cup\{\sigma\}$
  can be obtained in $k+1$ different ways from an unordered linearly independent $k$-element set, 
  depending on which of the elements plays the role of the product.
  An $r$-dimensional $\mF_2$-vector space has
  \[ (2^r-1)(2^r-2)\cdot\ldots\cdot(2^r-2^{k-1}) / k !\]
  linearly independent $k$-element subsets. Hence there are 
  \[ (2^r-1)(2^r-2)\cdot\ldots\cdot(2^r-2^{k-1})/(k+1)!\]
  minimally dependent subsets of $A^\circ$ with $k+1$ elements.
  So the total number of relations in the minimal presentation of $H(A,\star)$
  given by Theorem \ref{thm:Bredon_RO-graded} is
  \[ \sum_{k=2}^r\, \prod_{i=1}^k  \frac{2^r-2^{i-1}}{i+1}  \ .\]
  The following table  shows the number $2 (2^r-1)$ of polynomial generators
  of the representation-graded ring $H(A,\star)$,
  and the above number of minimal relations for small ranks:
  \begin{center}
\begin{tabular}{l|rrrrrrr}
    \text{rank$(A)$}          & 1 & 2 & 3 & 4 & 5 & 6 & 7 \\
    \hline
    \text{no.\ of variables} & 2 & 6 & 14 & 30 & 62 & 126 & 254 \\
    \text{no.\ of relations}   & 0 & 1 & 14 & 308 & 20,336 & 4,994,472 & 4,610,816,280
\end{tabular}
  \end{center}
While the number of relations grows very quickly with the rank, this is in a sense due to
the growing number of automorphisms.
Indeed, the automorphism group of $A$ acts transitively
on the minimally dependent sets of nontrivial $A$-characters of fixed cardinality.
So up to automorphisms of $A$, there is only one relation of degree $k+1$
for all $2\leq k\leq\text{rank}(A)$.\end{rk}

\section{Bredon homology with pre-Euler and inverse Thom classes inverted} \label{sec:localizations}

In this section we study certain localizations of the representation-graded
Bredon homology ring $H(A,\star)$.
We fix a subgroup $B$ of $A$ and consider the ring
obtained by inverting all pre-Euler classes that restrict nontrivially to $B$,
and all inverse Thom classes that restrict trivially to $B$.
Our presentation of $H(A,\star)$ yields a presentation
of the localized ring, see Theorem \ref{thm:localization}.

The localization by inverting all pre-Euler classes is also known as the
geometric fixed point ring of the Eilenberg-MacLane spectrum $H\underline\mF_2$;
Holler and Kriz \cite{holler-kriz} previously obtained a presentation
of it, which we recover as the special case $B=A$.
The mixed localizations were considered by Balmer and Gallauer \cite{balmer-gallauer};
we improve their \cite[Theorem 8.13]{balmer-gallauer},
for the prime $2$, from a `presentation modulo nilpotence' to
an actual presentation, see Remark \ref{rk:balmer-gallauer}.

\begin{construction}[Localizations]\label{con:invert classes}
  We let $B$ be a subgroup of an elementary abelian 2-group $A$.
  We let $H(A|B)$ be the integer-graded part of the localization
  of the representation-graded ring $H(A,\star)$ obtained by inverting the following classes:
  \begin{itemize}
  \item all classes $a_\lambda$ for all $A$-characters $\lambda$ such that $\lambda|_B$
    is nontrivial, and
  \item all classes $t_\lambda$ for all $A$-characters $\lambda$ such that $\lambda|_B$
    is trivial.
  \end{itemize}
  Every element of $H_k(A|B)$ is then of the form $x/a_V t_W$
  for some $A$-representation $V$ with $V^B=0$,
  some $A$-representation $W$ with $W^A=0$ on which $B$ acts trivially,
  and some $x\in H_{k+\dim(W)}(A,V\oplus W)$.
  These elements satisfy the relations
  \[ x/a_Vt_W \ = \  (x\cdot a_{\bar V} t_{\bar W})/a_{V\oplus \bar V}t_{W\oplus \bar W} \]
  for all $A$-representations $\bar V$ and $\bar W$  with the corresponding properties.
  The localizations introduce enough graded units so that the representation-grading
  effectively collapses to an integer-grading. In other words, we are not losing any
  information by restricting attention to the integer-graded subrings of the localizations.
\end{construction}  

The following theorem `localizes' the presentation of
the representation-graded ring $H(A,\star)$ from Theorem \ref{thm:Bredon_RO-graded}
to a presentation of the integer-graded ring $H(A|B)$.
A noticeable feature is that those relations in the presentation of $H(A,\star)$
that involve minimally dependent sets of four or more characters
become redundant in the localization.

\begin{theorem}\label{thm:localization}
  For every subgroup $B$ of an elementary abelian 2-group $A$,
  the graded ring $H(A|B)$ is a domain.
  The ring $H(A|B)$ is generated by
  \begin{itemize}
  \item 
    the homogeneous elements $x_\lambda=t_\lambda/a_\lambda$ of degree $1$,
    for $\lambda\in A^\circ$ with $\lambda|_B\ne 1$, and
  \item the homogeneous elements $e_\lambda=a_\lambda/t_\lambda$
    of degree $-1$, for $\lambda\in A^\circ$ with $\lambda|_B= 1$.
\end{itemize}
The ideal of relations between these generators is generated by
    the following polynomials:
    \begin{itemize}
    \item $x_\alpha x_\beta + x_\alpha x_\gamma + x_\beta x_\gamma$ for all triples
      of nontrivial $A$-characters such that $\alpha\cdot\beta\cdot\gamma=1$
      and such that $\alpha$, $\beta$ and $\gamma$ are nontrivial on $B$;      
    \item $x_\alpha+ x_\beta + x_\alpha x_\beta e_\gamma$ for all triples
      of nontrivial $A$-characters such that $\alpha\cdot\beta\cdot\gamma=1$,
      such that $\alpha$ and $\beta$ are nontrivial on $B$, and $\gamma|_B=1$;
    \item $e_\alpha+ e_\beta + e_\gamma$ for all triples
      of nontrivial $A$-characters such that $\alpha\cdot\beta\cdot\gamma=1$
      and $\alpha|_B=\beta|_B=\gamma|_B=1$.
    \end{itemize}
  \end{theorem}
\begin{proof}
  The multigraded ring $H(A,\star)$ is a domain by Theorem \ref{thm:domain},
  hence so is any localization at a multiplicative subset of homogeneous elements.
  Since $H(A|B)$ is a subring of such a localization, it is a domain, too.

  We write $\chi_+=\{\lambda\in A^\circ\colon \lambda|_B\ne 1\}$
  and $\chi_-=\{\lambda\in A^\circ\colon \lambda|_B= 1\}$.
  If we take the presentation of $H(A,\star)$
  given by Theorem \ref{thm:Bredon_RO-graded}, invert the relevant pre-Euler and
  inverse Thom classes, and restrict to integer-gradings,
  we recognize $H(A|B)$ as the quotient of the polynomial ring
  $\mF_2[x_\lambda, e_\mu\colon \lambda\in \chi_+,\mu\in\chi_-]$
  by the ideal generated by the polynomials
  \begin{align*}
    \bar r(T)\ &= \ r(T)/(\prod_{\mu\in T\cap\chi_+} a_\mu \cdot \prod_{\mu\in T\cap \chi_-} t_\mu) \\
    &= \ (\sum_{\lambda\in T\cap \chi_+} \prod_{\mu\in (T\cap\chi_+)\setminus\{\lambda\}} x_\mu)\
  + \ (\sum_{\lambda\in T\cap \chi_-} e_\lambda )\cdot\!\prod_{\mu\in T\cap\chi_+} x_\mu
  \end{align*}
  for all minimally dependent subsets $T$ of $A^\circ$.
  Among these are the minimally dependent subsets $T=\{\alpha,\beta,\gamma\}$
  that have three elements.
  We note that because $\alpha\cdot\beta\cdot\gamma=1$,
  whenever two of $\alpha$, $\beta$ and $\gamma$ are trivial on $B$, then so is the third.
  So the relations for  minimally dependent subsets with three elements
  are the ones from the statement of the theorem.
  
  In the rest of the proof we show that the polynomials $\bar r(T)$
  for  minimally dependent subsets $T$ with more than three elements
  are in the ideal generated by those with three elements.
  We argue by induction on the cardinality of $T$,
  and we let $T$ be a minimally dependent subset of $A^\circ$ with at least 4 elements.
  We pick two distinct elements $\alpha\ne \beta$ from $T$;
  by minimality, the product $\gamma=\alpha\cdot\beta$ then does not belong to $T$.
  We set $S=T\setminus\{\alpha,\beta\}$. We claim that 
  the polynomial $\bar r(T)$ lies in the ideal generated by
  $\bar r(S\cup\{\gamma\})$ and  $\bar r(\{\alpha,\beta,\gamma\})$.
  The sets $S\cup\{\gamma\}$ and $\{\alpha,\beta,\gamma\}$ are both minimally
  dependent, and both have fewer elements than $T$.
  So by induction, $\bar r(T)$ belongs to the ideal generated by the ternary relations.

  It remains to prove the claim. We distinguish two cases,
  depending on whether the restriction of $\gamma$ to $B$ is trivial or not.
  If $\gamma|_B= 1$, we exploit the polynomial relation 
  \[ t_{\gamma}\cdot r(T)\ = \   t_\alpha t_\beta\cdot r(S\cup\{\gamma\})
   \ + \ r(\{\alpha,\beta,\gamma\})\cdot \prod_{\lambda\in S}t_\lambda \]
 that holds by inspection.
 Because the class $t_\gamma$ is among those being inverted,
 this relation shows that the polynomial $\bar r(T)$ lies in the ideal generated by
 $\bar r(S\cup\{\gamma\})$ and $\bar r(\{\alpha,\beta,\gamma\})$.
 If $\gamma|_B\ne 1$, we exploit the relation 
 \[ a_{\gamma}\cdot r(T)\ = \
   (a_\alpha t_\beta+t_\alpha a_\beta)\cdot r(S\cup\{\gamma\})\ + \
   r(\{\alpha,\beta,\gamma\})\cdot r(S)  \]
 that also holds by inspection.
 Because the class $a_\gamma$ is among those being inverted,
 this relation shows that the polynomial $\bar r(T)$ lies in the ideal generated by
 $\bar r(S\cup\{\gamma\})$ and  $\bar r(\{\alpha,\beta,\gamma\})$.
\end{proof}

\begin{rk}[Relation to the work of Balmer and Gallauer]\label{rk:balmer-gallauer}
  In \cite{balmer-gallauer},
  Balmer and Gallauer also study the representation-graded Bredon homology ring
  $H(A,\star)$ as input for their computation of the Balmer spectrum of the tt-category of
  permutation modules;
  this ring is called the {\em twisted cohomology ring} of $A$
  in \cite[Definition 3.16]{balmer-gallauer}, and denoted $H^{\bullet\bullet}(A)$.
  The work of Balmer-Gallauer also covers elementary abelian $p$-groups
  for odd primes $p$, which we do not consider.
  The connection to Bredon homology comes from the fact that the homotopy category
  of permutation $A$-modules is equivalent to the homotopy category of modules in
  genuine $A$-spectra over the Eilenberg-MacLane spectrum $H\underline\mF_2$
  for the constant Mackey functor. 
  Under this equivalence, the invertible object $u_\lambda$
  introduced in \cite[Definition 3.3]{balmer-gallauer}
  corresponds to the representation sphere $S^\lambda$.
  Using our notation, Balmer-Gallauer prove the relation
  $ a_\alpha t_\beta  t_\gamma + t_\alpha a_\beta t_\gamma + t_\alpha t_\beta a_\gamma = 0$
  in \cite[Lemma 8.4]{balmer-gallauer},
  for all triples of $A$-characters whose product is 1.

  The localization $H(A|B)$ is introduced as $\mathcal O^\bullet_A(B)$
  in \cite[Definition 5.9]{balmer-gallauer}.
  In \cite[Construction 8.5]{balmer-gallauer},
  Balmer-Gallauer define a ring $\underline{\mathcal O}_A^\bullet(B)$
  by generators and relations as in the presentation of Theorem \ref{con:invert classes}.
  Then they show in \cite[Theorem 8.13]{balmer-gallauer}
  that the morphism $\underline{\mathcal O}_A^\bullet(B)\to \mathcal O_A^\bullet(B)$
  becomes an isomorphism after modding out the respective nilradicals.
  In the notation of Balmer-Gallauer,
  the content of our Theorem \ref{con:invert classes} is that
  $\underline{\mathcal O}_A^\bullet(B)\to \mathcal O_A^\bullet(B)$
  is already an isomorphism, without the need to divide out any ideal.
  In particular, our theorem confirms the expectation formulated in
  \cite[Remark 8.15]{balmer-gallauer}.
\end{rk}

For every genuine equivariant ring spectrum, the localization of the representation-graded
homotopy ring at the pre-Euler classes of all nontrivial irreducible representations
is isomorphic to the so-called {\em geometric fixed point ring},
see \cite[Proposition 3.20]{greenlees-may:eqstho}.
In particular, for every elementary abelian 2-group $A$,
the localization $H(A|A)$ is isomorphic to the geometric fixed point ring
$\Phi^A_*(H\underline\mF_2)$ of the Eilenberg-MacLane ring spectrum
of the constant Mackey functor $\underline \mF_2$.
Holler and Kriz gave a presentation of this ring in \cite[Theorem 2]{holler-kriz}. 
Their result is the special case $A=B$ of Theorem \ref{thm:localization}:

\begin{cor}\cite[Theorem 2]{holler-kriz}\label{cor:holler-kriz}
  Let $A$ be an elementary abelian 2-group.
  The graded ring $\Phi^A_*(H\underline\mF_2)=H(A|A)$ is generated by
  the homogeneous elements $x_\lambda=t_\lambda/a_\lambda$ of degree $1$,
  for all nontrivial $A$-characters $\lambda$.
  The ideal of relations between these generators is generated by
   the quadratic polynomials
   \[ x_\alpha x_\beta + x_\alpha x_\gamma + x_\beta x_\gamma \]
   for all triples
   of nontrivial $A$-characters such that $\alpha\cdot\beta\cdot\gamma=1$.
\end{cor}

\begin{rk} \label{rk:domain}
  As observed and expanded upon by S.\,Kriz in \cite{skriz}, the ring $H(A|A)$ agrees
  with the `reciprocal plane of the arrangement of all non-trivial hyperplanes of $A$'
  previously studied in the algebra literature.
  We refer to \cite{skriz} and the references therein for more information on this point of view.
  The identification with the reciprocal plane in particular implies that $H(A|A)$ is a domain.
  The localization maps $H_m(A,W)\to H_m(A|A)$ are injective,
  see \cite[Theorem 5 (ii)]{holler-kriz} or our Proposition \ref{prop:fundamental} (ii);
  so $H(A,\star)$ is a domain, too, 
  which we proved more directly in Theorem \ref{thm:domain}.
\end{rk}

If we specialize Theorem \ref{thm:localization} to the other extreme $B=\{1\}$,
we see that $H(A|\{1\})$ is generated by the Euler classes $e_\lambda=a_\lambda/t_\lambda$
for all nontrivial $A$-characters $\lambda$, and the ideal of relations is
generated by the linear polynomials $e_\alpha+e_\beta=e_\gamma$
for all triples of nontrivial $A$-characters such that $\alpha\cdot\beta=\gamma$.
So if $A$ has rank $r$, and $\lambda_1,\dots,\lambda_r$ is a basis of the vector space
of $A$-characters, then already the Euler classes $e_{\lambda_1},\dots,e_{\lambda_r}$
generate $H(A|\{1\})$, and there are no further relations between these.
In other words, the ring $H(A|\{1\})$ obtained by inverting all
inverse Thom classes is an $\mF_2$-polynomial algebra in the classes 
$e_{\lambda_1},\dots,e_{\lambda_r}$, which agrees with the group cohomology ring $H^*(A;\mF_2)$.

\end{document}